\documentclass[a4paper, 12pt]{amsart}%
\usepackage{amsmath}
\usepackage{graphicx}%
\usepackage{amsfonts}%
\usepackage{amssymb}

\newtheorem{theorem}{Theorem}[section]

\newtheorem{proposition}[theorem]{Proposition}
\newtheorem{corollary}[theorem]{Corollary}
\theoremstyle{definition}
\newtheorem{definition}[theorem]{Definition}
\newtheorem{remark}[theorem]{Remark}
\newtheorem{example}[theorem]{Example}

\begin{document}
\date{2011-8-15}
\title[Disjointness and unique ergodicity]{Disjointness and unique ergodicity of C*-dynamical systems}
\author{Rocco Duvenhage and Anton Str\"{o}h}
\address{Department of Physics\\
University of Pretoria\\
Pretoria 0002\\
South Africa}
\email{rocco.duvenhage@up.ac.za}
\address{Department of Mathematics and Applied Mathematics\\
University of Pretoria\\
Pretoria 0002\\
South Africa}
\email{anton.stroh@up.ac.za}

\begin{abstract}
We study an ergodic theorem for disjoint C*-dynamical systems, where
disjointness here is a noncommutative version of the concept introduced by
Furstenberg for classical dynamical systems. This is applied to W*-dynamical
systems. We also consider specific examples of disjoint C*-dynamical and
W*-dynamical systems, including for actions of other groups than $\mathbb{Z}$.
Unique ergodicity and unique ergodicity relative to the fixed point algebra
are closely related to disjointness, and are used to give examples of disjoint
C*-dynamical systems.

\end{abstract}
\maketitle

\section{Introduction}

In \cite{D, D2, D3} we studied joinings of W*-dynamical systems, generalizing
aspects of the classical case \cite{F67, G}. In particular disjointness and
some of its implications were studied. This included the relative case of
disjointness over a common factor of the W*-dynamical systems. Here we
continue studying disjointness and its consequences, in the context of
C*-dynamical systems possessing an invariant state. In the process we
illustrate how joining techniques can be applied to C*- and W*-dynamical systems.

In particular we derive an ergodic theorem for two disjoint C*-dynamical
systems. This theorem is then applied to W*-dynamical systems. A number of
results on ergodic theorems for C*- and W*-dynamical systems have appeared
recently. See for example \cite{NSZ}, \cite{D4}, \cite{BDS}, \cite{AET} and
\cite[Section 4]{EK}. The ergodic theorems we obtain here for pairs of
disjoint C*- and W*-dynamical systems, fit into this broader research effort,
and provide evidence that joinings are useful in this field.

A number of examples of disjoint systems illustrating these results, including
for the relative case, are constructed using $\mathbb{Z}$-actions on group von
Neumann algebras, as well as $\mathbb{R}$ and $\mathbb{R}^{2}$ actions on
quantum tori and tensor products of quantum tori.

A concept closely related to disjointness is unique ergodicity relative to the
fixed point algebra of a C*-dynamical system. Abadie and Dykema introduced the
latter notion for $\mathbb{Z}$-actions in their paper \cite{AD} as a
generalization of unique ergodicity. Certain aspects of unique ergodicity in
the context of C*-dynamical systems have already been explored by Avitzour
\cite{A1, A2} and Longo and Peligrad \cite{LP} for general group actions. In
classical ergodic theory the notion of unique ergodicity goes back to the
papers \cite{KB} and \cite{O}. Here we consider unique ergodicity relative to
the fixed point algebra for actions of $\sigma$-compact locally compact
amenable groups, and use it to obtain examples of disjoint C*-dynamical systems.

The main discussion of disjointness and its implications appears in Section 3.
Unique ergodicity and unique ergodicity relative to the fixed point algebra,
and their connection to disjointness, are studied in Sections 4 and 5
respectively. In particular we extend \cite[Theorem 3.2]{AD} to the more
general groups as mentioned above. However, we first discuss some of the
definitions and tools we use in the paper, in particular background regarding
the quantum torus and the Bochner integral, in the next section.

\section{Basic definitions and tools}

Here we discuss some of the definitions and an example that we will use in the
rest of the paper and also briefly discuss an important tool we need to study
unique ergodicity, namely the Bochner integral. We start with the latter.

The Bochner integral is described carefully in \cite[Chapter II]{DU} and we
refer the reader to this source for the standard definitions and results
regarding this topic. We only discuss some facts specific to our needs. The
Bochner integral is designed to integrate a $\rho$-measurable Banach space
valued function over a measure space $\left(  \Lambda,\Sigma,\rho\right)  $,
and the function is called Bochner integrable when its integral exists. In our
case the Banach space will be a unital C*-algebra $A$ and the measure space
will be finite.

First of all we note that if $\Lambda$ is a compact topological space and
$\Sigma$ its Borel $\sigma$-algebra, then every continuous function
$f:\Lambda\rightarrow X$ with $X$ a Banach space, is $\rho$-measurable. This
follows from the fact that $f(\Lambda)$ is a compact metric space and
therefore separable, which allows us to apply Pettis's measurability theorem
\cite{Pet} (see also \cite[Theorem II.1.2]{DU}) to conclude that $f$ is $\rho$-measurable.

Next consider a Bochner integrable function $f:\Lambda\rightarrow A$ where $A$
is a C*-algebra. Then
\begin{equation}
\left(  \int_{\Lambda}fd\rho\right)  ^{\ast}=\int_{\Lambda}f^{\ast}
d\rho\tag{2.1}%
\end{equation}
where $f^{\ast}$ is defined pointwise, i.e. $f^{\ast}(a):=f(a)^{\ast}$. This
is the case since by the definition of the Bochner integral there is a
sequence of simple functions $f_{n}:\Lambda\rightarrow A$ such that
$\lim_{n\rightarrow\infty}\int_{\Lambda}\left\|  f_{n}-f\right\|  d\rho=0$,
where here the norm of $A$ is composed with the function $f_{n}-f$, and then
$\int_{\Lambda}fd\rho=\lim_{n\rightarrow\infty}\int_{\Lambda}f_{n}d\rho$. So
clearly $\lim_{n\rightarrow\infty}\int_{\Lambda}\left\|  f_{n}^{\ast}-f^{\ast
}\right\|  d\rho=0$ which means $f^{\ast}$ is Bochner integrable and
$\int_{\Lambda}f^{\ast}d\rho=\lim_{n\rightarrow\infty}\int_{\Lambda}%
f_{n}^{\ast}d\rho$ from which (2.1) follows, since the $f_{n}$ are simple.

A similar argument can be used to prove many properties of the Bochner
integral. For example, if $G$ is a locally compact group with right Haar
measure $\rho$, and $f:\Lambda\rightarrow X$ is Bochner integrable where
$\Lambda\subset G$ is compact and $X$ a Banach space, then
\begin{equation}
\int_{\Lambda}f(gh)dg=\int_{\Lambda h}f(g)dg \tag{2.2}%
\end{equation}
for every $h\in G$, and where we wrote $dg$ as shorthand for $d\rho(g)$, as we
will do often in the rest of the paper when integrating with respect to a
right Haar measure. We will use (2.2) in Section 5.

Now we define the dynamical systems that we will be working with.

\begin{definition}
A \emph{C*-dynamical system} $\left(  A,\alpha\right)  $ consists of a unital
C*-algebra $A$ and an action $\alpha$ of a group $G$ on $A$ as $\ast
$-automorphisms. We use the notation $\alpha_{g}$ for an element of the group
action. If $B$ is an $\alpha$-invariant C*-subalgebra of $A$ containing the
unit of $A$, in other words $\alpha_{g}\left(  B\right)  =B$ for all $g\in G$,
we can define $\beta_{g}:=\alpha_{g}|_{B}$ to obtain a C*-dynamical system
$\left(  B,\beta\right)  $ called a \emph{factor} of $\left(  A,\alpha\right)
$. The \emph{fixed point algebra} of a C*-dynamical system $\left(
A,\alpha\right)  $ is defined as
\[
A^{\alpha}:=\left\{  a\in A:\alpha_{g}(a)=a\text{ for all }g\in G\right\}
\text{.}%
\]
Note that this gives an example of a factor of $\left(  A,\alpha\right)  $,
namely $\left(  A^{\alpha},\text{id}\right)  $ where id is the identity map.
\end{definition}

The ``term'' factor in the noncommutative context already appeared in
\cite{A1}, and generalizes the concept of a factor in classical topological dynamics.

In any situation in the rest of the paper involving more than one C*-dynamical
system, all the systems will be assumed to make use of actions of the same
group $G$. A simple but important construction in our work is the tensor
product of two C*-dynamical systems $\left(  A,\alpha\right)  $ and $\left(
B,\beta\right)  $, namely the C*-dynamical system $\left(  A\otimes
_{m}B,\alpha\otimes_{m}\beta\right)  $ where $A\otimes_{m}B$ is the maximal
C*-algebraic tensor product of $A$ and $B$, and $\left(  \alpha\otimes
_{m}\beta\right)  _{g}:=\alpha_{g}\otimes_{m}\beta_{g}$ for all $g\in G$.

We will in fact focus on a specific type of C*-dynamical system defined as follows:

\begin{definition}
A C*-dynamical system $\left(  A,\alpha\right)  $ is called \emph{amenable} if
$G$ is an amenable $\sigma$-compact locally compact group with right Haar
measure $\rho$ and the function $\Lambda\rightarrow A:g\mapsto\alpha_{g}(a)$
is $\rho$-measurable for every compact subset $\Lambda$ of $G$ and every $a\in
A$.
\end{definition}

Remember that an amenable $\sigma$-compact locally compact group $G$ has a
right F\o lner sequence $\left(  \Lambda_{n}\right)  $ consisting of compact
subsets of $G$ with strictly positive right Haar measure. This means that
\begin{equation}
\lim_{n\rightarrow\infty}\frac{\left|  \Lambda_{n}\bigtriangleup\left(
\Lambda_{n}g\right)  \right|  }{\left|  \Lambda_{n}\right|  }=0 \tag{2.3}%
\end{equation}
for every $g\in G$, where we use the notation $\left|  \Lambda\right|
=\rho(\Lambda)$. If in (2.3) we were to replace $\Lambda_{n}g$ by
$g\Lambda_{n}$ and we work in terms of the left Haar measure, then $\left(
\Lambda_{n}\right)  $ is rather called a \emph{left} F\o lner sequence. One
can even choose each $\Lambda_{n}$ to be symmetric, i.e. $\Lambda_{n}%
^{-1}=\Lambda_{n}$. Note that this implies that when the group is unimodular
(i.e. when $\rho$ is also a left Haar measure) we always have a right F\o lner
sequence which is also a left F\o lner sequence, namely any F\o lner sequence
consisting of symmetric sets. Of course a F\o lner sequence need not be
symmetric in order to be both a right and left F\o lner sequence, for example
the sequence $\Lambda_{n}=\left\{  1,...,n\right\}  $ for the group
$\mathbb{Z}$. All these facts (and more) regarding F\o lner sequences can be
found in \cite[Theorems 1 and 2]{E2} and \cite[Theorem 4]{E1}, and will be
used in this paper.

In an amenable C*-dynamical system the function $g\mapsto\alpha_{g}(a)$ is
Bochner integrable on compact subsets of $G$ by \cite[Theorem II.2.2]{DU},
since $\left\|  \alpha_{g}(a)\right\|  =\left\|  a\right\|  $ is constant. In
particular $g\mapsto\alpha_{g}(a)$ is Bochner integrable on every set
$\Lambda_{n}$ from the F\o lner sequence.

In the rest of the paper the notation $\left(  \Lambda_{n}\right)  $ will
refer to a right F\o lner sequence consisting of compact subsets, each with
strictly positive right Haar measure, of the relevant group $G$.

A central notion in our work will be that of an invariant state:

\begin{definition}
Given a C*-dynamical system $\left(  A,\alpha\right)  $, a state $\mu$ on $A$
is called an \emph{invariant} state of $\left(  A,\alpha\right)  $, or
alternatively an $\alpha$\emph{-invariant} state, if $\mu\circ\alpha_{g}=\mu$
for all $g\in G$. In this case we say that $\mathbf{A}=\left(  A,\alpha
,\mu\right)  $ is a \emph{state preserving C*-dynamical system}.
\end{definition}

We end this section with a basic and standard example that we will use a
number of times in the paper, namely the quantum torus.

\begin{example}
Our discussion is based on the exposition of the quantum torus in
\cite[Section 5.5]{W}, although our notation and conventions are somewhat
different. We consider the classical torus $\mathbb{T}^{2}=\mathbb{R}%
^{2}/\mathbb{Z}^{2}$ and the Hilbert space $H:=L^{2}(\mathbb{T}^{2})$ with
respect to the normalized Haar measure on $\mathbb{T}^{2}$. With every
$\theta\in\mathbb{R}$ we associate a C*-algebraic quantum torus $A_{\theta}$
which is the C*-subalgebra of $B(H)$ generated by the set of operators
$\left\{  u,v\right\}  \subset B(H)$ defined by%
\[
\left(  uf\right)  (x,y):=e^{2\pi ix}f\left(  x,y+\theta/2\right)
\]
and
\[
\left(  vf\right)  \left(  x,y\right)  :=e^{2\pi iy}f\left(  x-\theta
/2,y\right)
\]
for all $f\in H$ and $\left(  x,y\right)  \in\mathbb{T}^{2}$. We will also
consider the von Neumann algebraic quantum torus $M_{\theta}=A_{\theta
}^{\prime\prime}$. We note that
\[
uv=e^{2\pi i\theta}vu
\]
so in general $A_{\theta}$ and $M_{\theta}$ are not abelian. We can set
$\Omega:=1\in H$ which is then a cyclic vector for $A_{\theta}$ and
$M_{\theta}$, and it can be verified using harmonic analysis on $\mathbb{T}%
^{2}$ that it is also separating for $A_{\theta}$ and $M_{\theta}$. Hence we
can define a faithful trace Tr on both $A_{\theta}$ and $M_{\theta}$ by
\[
\text{Tr}(a):=\left\langle \Omega,a\Omega\right\rangle
\]
for all $a\in M_{\theta}$. Of course, Tr is normal on $M_{\theta}$. We will
call Tr the \emph{canonical trace} of the quantum torus.

A very simple $\mathbb{R}^{2}$ action can be defined on both $A_{\theta}$ and
$M_{\theta}$ as follows: Set%
\[
T_{s,t}(x,y):=\left(  x+s,y+t\right)  \in\mathbb{T}^{2}%
\]
for all $\left(  x,y\right)  \in\mathbb{T}^{2}$ and then define
\[
U_{s,t}:H\rightarrow H:f\mapsto f\circ T_{s,t}%
\]
for all $\left(  s,t\right)  \in\mathbb{R}^{2}$. One can then check that this
leads to a well defined action $\tau$ of $\mathbb{R}^{2}$ as $\ast
$-automorphisms on both $A_{\theta}$ and $M_{\theta}$ defined by
\[
\tau_{s,t}(a):=U_{s,t}aU_{s,t}^{\ast}%
\]
for all $a$ in $A_{\theta}$ or $M_{\theta}$, and all $\left(  s,t\right)
\in\mathbb{R}^{2}$. Furthermore, Tr is an invariant state of this dynamics in
both the C*-algebraic and von Neumann algebraic cases. It can be shown that
\[
\mathbb{R}^{2}\rightarrow A_{\theta}:\left(  s,t\right)  \mapsto\tau_{s,t}(a)
\]
is norm continuous for every $a\in A_{\theta}$. So $\left(  A_{\theta}%
,\tau,\text{Tr}\right)  $ is an amenable state preserving C*-dynamical system.
A useful fact is that
\begin{equation}
\tau_{s,t}(u)=e^{2\pi is}u\text{ \ and \ }\tau_{s,t}(v)=e^{2\pi it}v \tag{2.4}%
\end{equation}
for all $\left(  s,t\right)  \in\mathbb{R}^{2}$. We will also consider
variations on this dynamics in the rest of the paper. The C*-algebraic case
will appear in Sections 4 and 5, while the von Neumann algebraic case will
appear in Section 3.

We note that at least in the case where $\theta$ is irrational, $A_{\theta}$
is nuclear. This follows for example from the fact that $A_{\theta}$ can be
written as a crossed product of the abelian (and therefore nuclear) C*-algebra
$C(\mathbb{T})$ by an action of the amenable group $\mathbb{Z}$ as explained
in \cite[Theorem VI.1.4 and Example VIII.1.1]{Dav}, combined with the fact
that such crossed products are nuclear \cite[Proposition 2.1.2]{RS}.
Nuclearity of $A_{\theta}$ will be useful for us when discussing certain
examples in Section 5.
\end{example}

\section{Disjointness}

Disjointness of C*-dynamical systems was first considered by Avitzour in
\cite{A2}, extending the concept of disjointness in topological dynamical
systems \cite[Part II]{F67}. In this section however our approach is from the
ergodic theory point of view, in that we define disjointness of a pair of
state preserving C*-dynamical systems in terms of invariant states on a tensor
product of the pair, similar to \cite{D, D2, D3}. It is based on the idea of
disjointness in classical ergodic theory, also originating in Furstenberg's
paper \cite{F67}, and treated extensively in \cite{G}. In Section 5 we will
also consider disjointness of certain C*-dynamical systems (without specified
invariant states), but still in terms of invariant states, rather than the
topological version of Avitzour.

\begin{definition}
Let $\mathbf{A}=\left(  A,\alpha,\mu\right)  $ and $\mathbf{B}=\left(
B,\beta,\nu\right)  $ be state preserving C*-dynamical systems. A
\emph{joining} of $\mathbf{A}$ and $\mathbf{B}$ is an invariant state $\omega$
of $\left(  A\otimes_{m}B,\alpha\otimes_{m}\beta\right)  $ such that
$\omega\left(  a\otimes1\right)  =\mu(a)$ and $\omega\left(  1\otimes
b\right)  =\nu(b)$ for all $a\in A$ and $b\in B$. The set of all joinings of
$\mathbf{A}$ and $\mathbf{B}$ is denoted by $J(\mathbf{A},\mathbf{B})$. If
$J(\mathbf{A},\mathbf{B})=\{\mu\otimes_{m}\nu\}$, then $\mathbf{A}$ and
$\mathbf{B}$ are called \emph{disjoint}. More generally we can consider a
factor $\left(  R,\rho\right)  $ of $\left(  A\otimes_{m}B,\alpha\otimes
_{m}\beta\right)  $, and a $\rho$-invariant state $\psi$ on $R$ which has at
least one extension to a joining of $\mathbf{A}$ and $\mathbf{B}$. So we
obtain a state preserving C*-dynamical system $\mathbf{R}=\left(  R,\rho
,\psi\right)  $. Denote by $J_{\mathbf{R}}\left(  \mathbf{A},\mathbf{B}%
\right)  $ the subset of elements $\omega$ of $J\left(  \mathbf{A}%
,\mathbf{B}\right)  $ such that $\omega|_{R}=\psi$. If $J_{\mathbf{R}}\left(
\mathbf{A},\mathbf{B}\right)  $ contains exactly one element, then we say that
$\mathbf{A}$ and $\mathbf{B}$ are \emph{disjoint with respect to} $\mathbf{R}$.
\end{definition}

In this definition one can equivalently define a joining of $\mathbf{A}$ and
$\mathbf{B}$ in terms of the algebraic tensor product $\left(  A\odot
B,\alpha\odot\beta\right)  $ instead of $\left(  A\otimes_{m}B,\alpha
\otimes_{m}\beta\right)  $, since a state on the algebraic tensor product can
be extended to a state on the maximal C*-algebraic tensor product; see for
example \cite[Proposition 4.1]{D2}. This fact is used later on. In the rest of
this section, the notation $\mathbf{A}$, $\mathbf{B}$ and $\mathbf{R}$ will
refer to triples as in Definition 3.1, and we also use the notation
$\mathbf{F}$ for the triple $\left(  F,\varphi,\lambda\right)  $, with the
symbol $G$ for the group always implied.

We also need the following weaker concept:

\begin{definition}
Let $A$ and $B$ be unital C*-algebras with states $\mu$ and $\nu$
respectively. A \emph{coupling} of the pairs $\left(  A,\mu\right)  $ and
$\left(  B,\nu\right)  $ is a state $\kappa$ on $A\otimes_{m}B$ such that
$\kappa\left(  a\otimes1\right)  =\mu(a)$ and $\kappa\left(  1\otimes
b\right)  =\nu(b)$ for all $a\in A$ and $b\in B$. If furthermore $\psi$ is a
state on a C*-subalgebra $R$ of $A\otimes_{m}B$ such that $\kappa|_{R}=\psi$,
then we call $\kappa$ a \emph{coupling} of $\left(  A,\mu\right)  $ and
$\left(  B,\nu\right)  $ \emph{with respect to} $\left(  R,\psi\right)  $.
\end{definition}

The basic result of this section is the following ergodic theorem, which can
be viewed as a generalization of \cite[Proposition 2.3]{D2}.

\begin{theorem}
Let $\mathbf{A}$ and $\mathbf{B}$ be state preserving amenable C*-dynamical
systems which are disjoint with respect to $\mathbf{R}$. Let $\left(
\kappa_{n}\right)  $ be a sequence of couplings of $\left(  A,\mu\right)  $
and $\left(  B,\nu\right)  $ with respect to $\left(  R,\psi\right)  $. Then
\begin{equation}
\lim_{n\rightarrow\infty}\frac{1}{\left|  \Lambda_{n}\right|  }\int
_{\Lambda_{n}}\kappa_{n}\circ\left(  \alpha_{g}\otimes_{m}\beta_{g}\right)
(c)dg=\omega(c) \tag{3.1}%
\end{equation}
for all $c\in A\otimes_{m}B$, where $\omega$ is the unique element of
$J_{\mathbf{R}}\left(  \mathbf{A},\mathbf{B}\right)  $. Alternatively, instead
of assuming that $g\mapsto\alpha_{g}(a)$ and $g\mapsto\beta_{g}(b)$ are $\rho
$-measurable on compact subsets of $G$ for all $a\in A$ and $b\in B$, we can
simply assume that for every $n$ the function $g\mapsto\kappa_{n}\circ\left(
\alpha_{g}\otimes_{m}\beta_{g}\right)  (a\otimes b)$ is measurable on compact
subsets of $G$ for all $a\in A$ and $b\in B$, and then the result still holds.
\end{theorem}

\begin{proof}
First note that from the definition of $\rho$-measurability \cite[Definition
II.1.1]{DU} it follows that $\Lambda_{n}\rightarrow A\otimes_{m}%
B:g\mapsto\alpha_{g}(a)\otimes\beta_{g}(b)$ is $\rho$-measurable for all $a\in
A$ and $b\in B$, and therefore weakly $\rho$-measurable, so in particular
$g\mapsto\kappa_{n}\left(  \alpha_{g}(a)\otimes\beta_{g}(b)\right)  $ is
measurable on $\Lambda_{n}$, which also follows from the alternative
assumption in the theorem. From Lebesgue's dominated convergence theorem it
then follows that $g\mapsto\kappa_{n}\circ\left(  \alpha_{g}\otimes_{m}%
\beta_{g}\right)  (c)$ is measurable on $\Lambda_{n}$ for all $c\in
A\otimes_{m}B$, which means that the integrals in (3.1) indeed exist.

We define a sequence of states $\left(  \omega_{n}\right)  $ on $A\otimes
_{m}B$ by
\[
\omega_{n}(c):=\frac{1}{\left|  \Lambda_{n}\right|  }\int_{\Lambda_{n}}%
\kappa_{n}\circ\left(  \alpha_{g}\otimes_{m}\beta_{g}\right)  (c)dg
\]
which then has a cluster point $\omega^{\prime}$ in the weak* topology in the
set of all states on $A\otimes_{m}B$. By an argument similar to that in the
proof of \cite[Proposition 2.3]{D2}, it follows from eq. (2.3) that
$\omega^{\prime}\in J(\mathbf{A},\mathbf{B})$. It is also easy to check, using
the $\rho$-invariance of $\psi$, that $\omega_{n}|_{R}=\psi$, hence
$\omega^{\prime}|_{R}=\psi$. So by the assumed disjointness, $\omega^{\prime
}=\omega$. Therefore $\omega$ is the unique cluster point of $\left(
\omega_{n}\right)  $ in the weak* topology, from which we conclude that
$\left(  \omega_{n}\right)$ converges to $\omega$ in the weak*
topology, proving the theorem.
\end{proof}

This theorem is quite abstract, but has a number of interesting consequences
as we will see. To begin to understand its meaning, and for later reference,
we state the following special case explicitly.

\begin{corollary}
Let $\mathbf{A}$ and $\mathbf{B}$ be disjoint state preserving amenable
C*-dynamical systems, and $\left(  \kappa_{n}\right)  $ a sequence of
couplings of $\left(  A,\mu\right)  $ and $\left(  B,\nu\right)  $. Then
\[
\lim_{n\rightarrow\infty}\frac{1}{\left|  \Lambda_{n}\right|  }\int
_{\Lambda_{n}}\kappa_{n}\circ\left(  \alpha_{g}\otimes\beta_{g}\right)
(c)dg=\mu\otimes_{m}\nu(c)
\]
for all $c\in A\otimes_{m}B$. Again we can also use the alternative
measurability assumptions given in Theorem 3.3
\end{corollary}

\begin{remark}
This theorem and its corollary in fact continues to hold if we use a F\o lner
net instead of a F\o lner sequence, and a net of couplings with the same
directed set as the F\o lner net. This allows one to use for example the form
$\lim_{T\rightarrow\infty}\frac{1}{T}\int_{-T}^{T}$ in the case where $T\in
G=\mathbb{R}$.
\end{remark}

The corollary is the one extreme of Theorem 3.3, namely when $R=\mathbb{C}1$.
The other extreme is when $R=A\otimes_{m}B$, in which case Theorem 3.3 becomes
trivial, since $\kappa_{n}=\psi=\omega$ is then $\alpha\otimes_{m}\beta
$-invariant. Broadly the idea is therefore to apply Theorem 3.3 to situations
where $R$ is a proper subalgebra of $A\otimes_{m}B$.

In the remainder of this section we illustrate Theorem 3.3 by studying
consequences for W*-dynamical systems, using results from \cite{D, D2, D3}. In
subsequent sections we will illustrate the C*-algebraic case itself, by way of examples.

\begin{definition}
A \emph{W*-dynamical system} is a state preserving C*-dynamical system
$\mathbf{A}=\left(  A,\alpha,\mu\right)  $ where $\mu$ is a faithful normal
state on a (necessarily $\sigma$-finite) von Neumann algebra $A$.
\end{definition}

We now also use Definitions 3.1 and 3.2 for W*-dynamical systems, since
W*-dynamical systems are state preserving C*-dynamical systems. However, in
\cite{D, D2, D3} joinings were expressed in terms of algebraic tensor
products, so as mentioned above, it is important to keep in mind that any
state on the algebraic tensor product of two C*-algebras, and in particular
two von Neumann algebras, can be extended to a state on the maximal
C*-algebraic tensor product of the two algebras, as explained in \cite[Section
4]{D2}. So even though we are working on the maximal C*-algebraic tensor
product, we will still be able to apply results from \cite{D, D2, D3}.

We need some additional background and notation regarding a W*-dynamical
system $\mathbf{A}$ before we proceed (a more general setup and more details
can be found in \cite{D, D2, D3}): The cyclic representation of $A$ obtained
from $\mu$ by the GNS construction will be denoted by $\left(  H,\pi
,\Omega\right)  $. The associated modular conjugation will be denoted by $J$
and we set
\[
j:B(H)\rightarrow B(H):a\mapsto Ja^{\ast}J.
\]
Note that $j^{-1}=j$. Then $\alpha$ can be represented by a unitary group $U$
on $H$ defined by extending
\[
U_{g}\pi(a)\Omega:=\pi(\alpha_{g}(a))\Omega.
\]
It satisfies
\[
U_{g}\pi(a)U_{g}^{\ast}=\pi(\alpha_{g}(a))
\]
for all $a\in A$ and $g\in G$; also see \cite[Corollary 2.3.17]{BR}. We are
particularly interested in the case where we have a second W*-dynamical system
$\mathbf{B}$ with $\left(  B,\nu\right)  =\left(  A,\mu\right)  $, in other
words we want to consider two dynamics on the same von Neumann algebra and
state. Exactly as for $U$ above, we obtain a unitary representation $V$ of
$\beta$.

Following the plan in \cite[Construction 3.4]{D} the ``commutant''
$\mathbf{\tilde{B}}$ of $\mathbf{B}$ can be defined as follows: Set $\tilde
{B}:=\pi(B)^{\prime}$ and then carry the state and dynamics of $\mathbf{B}$
over to $\tilde{B}$ in a natural way using $j$, by defining a state
$\tilde{\nu}$ and $\ast$-automorphism $\tilde{\beta}_{g}$ on $\tilde{B}$ by
$\tilde{\nu}(b):=\nu\circ\pi^{-1}\circ j(b)$ and $\tilde{\beta}_{g}%
(b):=j\circ\pi\circ\beta_{g}\circ\pi^{-1}\circ j(b)$ for all $g\in G$. From
Tomita-Takesaki theory one has that $V_{g}J=JV_{g}$ (see \cite[Construction
3.4]{D}). Then
\[
\tilde{\nu}(b)=\left\langle \Omega,b\Omega\right\rangle
\]
and
\[
\tilde{\beta}_{g}(b)=V_{g}bV_{g}^{\ast}%
\]
for all $b\in\tilde{B}$ and $g\in G$. This tells us that the unitary
representation of $\tilde{\beta}$ is the same as that of $\beta$, namely $V$.

A \emph{factor} of $\mathbf{A}$ is a W*-dynamical system $\mathbf{F=}\left(
F,\varphi,\lambda\right)  $ such that $\left(  F,\varphi\right)  $ is a factor
of $\left(  A,\alpha\right)  $ as in Definition 2.1, but with $F$ now a von
Neumann subalgebra of $A$, and $\lambda=\mu|_{F}$. In \cite{D, D2, D3} a more
general situation was considered, but this definition of a factor will do in
this paper. The term ``factor'' is used, since this is the terminology
appearing in the classical literature, and does not refer to a von Neumann
algebra with trivial center. (In \cite{D2} and \cite{D3} the term
``subsystem'' was used instead, but this terminology has a different, though
complementary, meaning in classical dynamics; see for example \cite[Section
1.1]{G}.) We say that $\mathbf{F}$ is a \emph{modular} factor of $\mathbf{A}$
if $F$ is invariant under the modular group $\sigma$ associated to $\left(
A,\mu\right)  $, i.e. $\sigma_{t}\left(  F\right)  =F$ for all $t\in
\mathbb{R}$. Assuming that $\mathbf{F}$ is also a modular factor of
$\mathbf{B}$ above, and writing
\[
\tilde{F}:=j\circ\pi\left(  F\right)
\]
we obtain a modular factor $\mathbf{\tilde{F}}=\left(  \tilde{F}%
,\tilde{\varphi},\tilde{\lambda}\right)  $ of $\mathbf{\tilde{B}}$. One can
then define a \emph{diagonal} state $\Delta_{\lambda}$ on the algebraic tensor
product $F\odot\tilde{F}$ by extending
\[
\delta\left(  a\otimes b\right)  :=\pi(a)b
\]
to a unital $\ast$-homomorphism $\delta:F\odot\tilde{F}\rightarrow B(H)$ and
setting%
\[
\Delta_{\lambda}(c):=\left\langle \Omega,\delta(c)\Omega\right\rangle
\]
for all $c\in F\odot\tilde{F}$. Since $\mathbf{F}$ is a modular factor, it
turns out that there is at least one joining of $\mathbf{A}$ and
$\mathbf{\tilde{B}}$ extending the state $\Delta_{\lambda}$, as explained in
\cite[Section 3]{D3}.

\begin{definition}
Let $\mathbf{A}$ and $\mathbf{B}$ be W*-dynamical systems with $\left(
A,\mu\right)  =\left(  B,\nu\right)  $, and assume both have $\mathbf{F}$ as a
modular factor. If $\Delta_{\lambda}$ has a unique extension to a joining of
$\mathbf{A}$ and $\mathbf{\tilde{B}}$, then we call $\mathbf{A}$ and
$\mathbf{\tilde{B}}$ \emph{disjoint over} $\mathbf{F}$.
\end{definition}

It is clear that this definition is closely related to Definition 3.1, and the
connection will be further clarified in the proof of Theorem 3.8 below.

Using Theorem 3.3 we can now prove the following result, where the notation
id$_{F}$ refers to the identity map $F\rightarrow F$.

\begin{theorem}
Let $\mathbf{A}$ and $\mathbf{B}$ be W*-dynamical systems with $\left(
A,\mu\right)  =\left(  B,\nu\right)  $ and with $G$ amenable, $\sigma$-compact
and locally compact. Let $\mathbf{F}$ be a modular factor of both $\mathbf{A}$
and $\mathbf{B}$, and assume that $\mathbf{A}$ and $\mathbf{\tilde{B}}$ are
disjoint over $\mathbf{F}$. Let $\eta_{n}:A\rightarrow A$ be any sequence of
$\ast$-automorphisms with $\eta_{n}|_{F}=$ id$_{F}$ and $\mu\circ\eta_{n}=\mu
$, and assume that $G\rightarrow\mathbb{C}:g\mapsto\mu\left(  \eta_{n}\left(
\alpha_{g}(a)\right)  \beta_{g}(b)\right)  $ is Borel measurable for all
$a,b\in A$. Then
\[
\lim_{n\rightarrow\infty}\frac{1}{\left|  \Lambda_{n}\right|  }\int
_{\Lambda_{n}}\mu\left(  \eta_{n}\left(  \alpha_{g}(a)\right)  \beta
_{g}(b)\right)  dg=\mu\left(  D(a)D(b)\right)
\]
for all $a,b\in A$, where $D:A\rightarrow F$ is the unique conditional
expectation such that $\mu\circ D=\mu$.

In particular, if $\mathbf{A}$ and $\mathbf{\tilde{B}}$ are disjoint (that is,
$F=\mathbb{C}$), then
\[
\lim_{n\rightarrow\infty}\frac{1}{\left|  \Lambda_{n}\right|  }\int
_{\Lambda_{n}}\mu\left(  \eta_{n}\left(  \alpha_{g}(a)\right)  \beta
_{g}(b)\right)  dg=\mu(a)\mu(b)
\]
for all $a,b\in A$.
\end{theorem}

\begin{proof}
The existence and uniqueness of $D$ follow from Tomita-Takesaki theory and the
fact that $\mathbf{F}$ is a modular factor of $\mathbf{A}$. Similarly we have
a unique conditional expectation $\tilde{D}:\tilde{A}\rightarrow\tilde{F}$
such that $\tilde{\mu}\circ\tilde{D}=\tilde{\mu}$. Note that if $P$ is the
projection of $H$ onto $\overline{\pi(F)\Omega}$, in terms of the notation
above, then $\pi(D(a))=P\pi(a)P$ and $\tilde{D}(b)=PbP$ for all $a\in A$ and
$b\in\tilde{B}$, as discussed for example in \cite[Subsection 10.2]{St}.

Let $R$ be the closure of $F\odot\tilde{F}$ in $A\otimes_{m}B$. Note that in
terms of the notation above, $\Delta_{\lambda}$ has at least one extension to
a state on $A\odot\tilde{B}$, namely the relatively independent joining
$\mu\odot_{\lambda}\tilde{\nu}=\Delta_{\lambda}\circ\left(  D\odot\tilde
{D}\right)  $ \cite[Section 3]{D3}, which can in turn be uniquely extended to
a state, say $\mu\otimes_{\lambda}\tilde{\nu}$, on $A\otimes_{m}\tilde{B}$
\cite[Proposition 4.1]{D2}, and therefore in particular $\Delta_{\lambda}$ can
be extended to a (necessarily unique) state on $R$. Since $\mathbf{A}$ and
$\mathbf{\tilde{B}}$ are in fact disjoint over $\mathbf{F}$, $\mu
\otimes_{\lambda}\tilde{\nu}$ is the unique joining extending $\Delta
_{\lambda}$. In the sense of Definition 3.1, $\mathbf{A}$ and $\mathbf{\tilde
{B}}$ are therefore disjoint with respect to $\mathbf{R}=\left(  R,\rho
,\Delta_{\lambda}\right)  $ where $\rho_{g}:=\alpha_{g}\otimes_{m}\tilde
{\beta}_{g}|_{R}$.

We define a coupling $\mu_{\bigtriangleup}$ of $\left(  A,\mu\right)  $ and
$\left(  \tilde{B},\tilde{\nu}\right)  $ by $\mu_{\bigtriangleup
}(c):=\left\langle \Omega,\delta(c)\Omega\right\rangle $ where the unital
$\ast$-homomorphism $\delta:A\otimes_{m}\tilde{B}\rightarrow B(H)$ is obtained
from $\delta(a\otimes b):=\pi(a)b$ for all $a\in A$ and $b\in\tilde{B}$,
extending $\delta$ appearing above. This type of coupling has also been used
in \cite{Fid}. From this we then define the sequence of couplings
\[
\kappa_{n}:=\mu_{\bigtriangleup}\circ\left(  \eta_{n}\otimes_{m}%
\text{id}_{\tilde{B}}\right)  .
\]
It is easily verified using elementary tensors, that each $\kappa_{n}$ is a
coupling with respect to $\left(  R,\Delta_{\lambda}\right)  $ in the sense of
Definition 3.2. Let $W_{n}$ be the unitary representation of $\eta_{n}$ in the
cyclic representation obtained form $\mu$, and note that $\mu\left(  \eta
_{n}\left(  \alpha_{g}(a)\right)  \beta_{g}(b)\right)  =\left\langle
\pi(a^{\ast})\Omega,U_{g}^{\ast}W_{n}^{\ast}V_{g}\pi(b)\Omega\right\rangle $
for all $a,b\in A$, so from Lebesgue's dominated convergence theorem and the
fact that $\pi(A)\Omega$ is dense in $H$ it follows that $g\mapsto\left\langle
x,U_{g}^{\ast}W_{n}^{\ast}V_{g}y\right\rangle $ is measurable on compact
subsets of $G$ for all $x,y\in H$. This in turn implies that $g\mapsto
\kappa_{n}\circ\left(  \alpha_{g}\otimes_{m}\tilde{\beta}_{g}\right)  \left(
a\otimes b\right)  =\left\langle \pi(a^{\ast})\Omega,U_{g}^{\ast}W_{n}^{\ast
}V_{g}b\Omega\right\rangle $ is measurable on compact subsets of $G$ for all
$a\in A$ and $b\in\tilde{B}$. Applying Theorem 3.3, it follows that
\begin{align*}
&  \lim_{n\rightarrow\infty}\frac{1}{\left|  \Lambda_{n}\right|  }%
\int_{\Lambda_{n}}\left\langle \pi(a)\Omega,U_{g}^{\ast}W_{n}^{\ast}%
V_{g}b\Omega\right\rangle dg\\
&  =\lim_{n\rightarrow\infty}\frac{1}{\left|  \Lambda_{n}\right|  }%
\int_{\Lambda_{n}}\kappa_{n}\circ\left(  \alpha_{g}\otimes_{m}\tilde{\beta
}_{g}\right)  \left(  a^{\ast}\otimes b\right)  dg\\
&  =\mu\otimes_{\lambda}\tilde{\nu}\left(  a^{\ast}\otimes b\right) \\
&  =\left\langle \pi(D(a))\Omega,\tilde{D}(b)\Omega\right\rangle \\
&  =\left\langle P\pi(a),Pb\Omega\right\rangle
\end{align*}
and therefore, since $\tilde{B}\Omega$ is also dense in $H$,
\[
\lim_{n\rightarrow\infty}\frac{1}{\left|  \Lambda_{n}\right|  }\int
_{\Lambda_{n}}\left\langle x,U_{g}^{\ast}W_{n}^{\ast}V_{g}y\right\rangle
=\left\langle Px,Py\right\rangle
\]
for all $x,y\in H$. Finally we can apply this to our current problem, namely
for all $a,b\in A$ we have
\begin{align*}
&  \lim_{n\rightarrow\infty}\frac{1}{\left|  \Lambda_{n}\right|  }%
\int_{\Lambda_{n}}\mu\left(  \eta_{n}\left(  \alpha_{g}(a)\right)  \beta
_{g}(b)\right)  dg\\
&  =\lim_{n\rightarrow\infty}\frac{1}{\left|  \Lambda_{n}\right|  }%
\int_{\Lambda_{n}}\left\langle \pi(a^{\ast})\Omega,U_{g}^{\ast}W_{n}^{\ast
}V_{g}\pi(b)\Omega\right\rangle dg\\
&  =\left\langle P\pi(a^{\ast})\Omega,P\pi(b)\Omega\right\rangle \\
&  =\left\langle \pi(D(a^{\ast}))\Omega,\pi(D(b))\Omega\right\rangle \\
&  =\left\langle \Omega,\pi(D(a)D(b))\Omega\right\rangle \\
&  =\mu(D(a)D(b))
\end{align*}
as required.

In particular, when $\mathbf{A}$ and $\mathbf{\tilde{B}}$ are disjoint, we
have $D=\mu$, so $\mu(a)\mu(b)$. Alternatively we can follow the above proof,
but using Corollary 3.4 instead of Theorem 3.3, to obtain this special case.
\end{proof}

We are therefore interested in pairs of disjoint W*-dynamical systems. An
archetypal example is weak mixing versus compactness; see \cite{NSZ, D2, D3}
for an extended discussion and definitions. In particular we recall from
\cite[Section 5]{D3} that a W*-dynamical system with $G$ abelian, has a
biggest compact factor, which we denote by 
$\mathbf{A}^{K}=\left(A^{K},\alpha^{K},\mu^K\right)$, 
and this factor is necessarily modular.

\begin{corollary}
Let $\mathbf{A}$ and $\mathbf{B}$ be W*-dynamical systems with $\left(
A,\mu\right)  =\left(  B,\nu\right)  $. Assume that $G$ is abelian, amenable,
$\sigma$-compact and locally compact, and that $G\rightarrow\mathbb{C}%
:g\mapsto\mu\left(  \alpha_{g}(a)\beta_{g}(b)\right)  $ is Borel measurable
for all $a,b\in A$. Assume furthermore that $\mathbf{B}$ is compact and that
$\mathbf{A}^{K}$ is a factor of $\mathbf{B}$. Then%
\[
\lim_{n\rightarrow\infty}\frac{1}{\left|  \Lambda_{n}\right|  }\int
_{\Lambda_{n}}\mu\left(  \alpha_{g}(a)\beta_{g}(b)\right)  dg=\mu\left(
D(a)D(b)\right)
\]
for all $a,b\in A$, where $D:A\rightarrow A^{K}$ is the unique conditional
expectation such that $\mu\circ D=\mu$.

In particular, if $\mathbf{A}$ is weakly mixing (in which case $\mathbf{A}%
^{K}$ is trivial, that is, $A^{K}=\mathbb{C}1$), then%
\begin{equation}
\lim_{n\rightarrow\infty}\frac{1}{\left|  \Lambda_{n}\right|  }\int
_{\Lambda_{n}}\mu\left(  \alpha_{g}(a)\beta_{g}(b)\right)  dg=\mu(a)\mu(b)
\tag{3.2}%
\end{equation}
for all $a,b\in A$.
\end{corollary}

\begin{proof}
Since $\nu=\mu$, while $\mathbf{A}^{K}$ is a modular factor of $\mathbf{A}$,
it follows that $\mathbf{A}^{K}$ is also a modular factor of $\mathbf{B}$.
Therefore $\mathbf{A}$ and $\mathbf{\tilde{B}}$ are disjoint over
$\mathbf{A}^{K}$ by \cite[Theorem 5.6]{D3}, and the result follows from Theorem
3.8. When $\mathbf{A}$ is weakly mixing, then $A^{K}=\mathbb{C}1$ by
\cite[Corollary 5.5]{D3}, therefore $D=\mu$, and the special case follows.
\end{proof}

\begin{remark}
Using \cite[Theorem 2.8]{D2}, we can obtain a variation on the special case in
the result above, namely assuming that $\mathbf{B}$ is ergodic with discrete
spectrum, instead of compact, but relaxing the assumption that $G$ be abelian,
eq. (3.2) still holds.
\end{remark}

\begin{remark}
Following a similar proof, using \cite[Theorem 4.3]{D3} instead of
\cite[Theorem 5.6]{D3}, one finds that Corollary 3.9 also holds when
$\mathbf{A}$ is any W*-dynamical system, $\mathbf{B}$ is an identity system
(i.e. $\beta_{g}=$ id$_{B}$), and $\mathbf{A}^{K}$ is replaced by the fixed
point factor $\mathbf{A}^{\alpha}=\left(  A^{\alpha},\text{id},\mu
|_{A^{\alpha}}\right)  $ of $\mathbf{A}$, without $G$ having to be abelian.
Note that the fixed point algebra indeed gives a modular factor $\mathbf{A}%
^{\alpha}$ of $\mathbf{A}$ \cite[Proposition 4.2]{D3}. The special case eq.
(3.2) now holds for ergodic $\mathbf{A}$, which means that $\mathbf{A}%
^{\alpha}=\mathbb{C}1$, and is then a standard characterization of ergodicity
(see also Section 4). This version of Corollary 3.9 however also follows
directly from the mean ergodic theorem and the representation of the
conditional expectations in terms of Hilbert space projections, so does not
require joining techniques for its proof.
\end{remark}

We end this section with a discussion of examples to illustrate the
W*-dynamical results above. Simple examples of weakly mixing (in fact strongly
mixing) and compact W*-dynamical systems on the same algebra and state, for
$\mathbb{Z}$-actions, are provided by group von Neumann algebras; see
\cite[Theorems 3.4 and 3.6]{D2}. We now firstly look at the relative case of
this example as well, namely where $\mathbf{A}^{K}$ is not trivial.

We use the same basic setting as in \cite[Section 3]{D2}, namely we consider
an automorphism $T$ of an arbitrary group $\Gamma$ to which we assign the
discrete topology. This leads to a \emph{dual system} $\mathbf{A}=\left(
A,\alpha,\mu\right)  $ with $G=\mathbb{Z}$, where $A$ is the group von Neumann
algebra of $\Gamma$, $\mu$ is its canonical trace, and $\alpha\left(
a\right)  :=UaU^{\ast}$ for all $a\in A$, with the unitary operator
$U:H\rightarrow H$ defined by
\[
Uf:=f\circ T^{-1}%
\]
where $H:=\ell^{2}(\Gamma)$. For any $g\in\Gamma$ we define $\delta_{g}\in H$
by $\delta_{g}(g)=1$ and $\delta_{g}(h)=0$ for $h\neq g$. Note that
$U\delta_{g}=\delta_{Tg}$. Setting $\Omega:=\delta_{1}$ where $1$ here refers
to the identity element of $\Gamma$, it is then easily seen that $\left(
H,\text{id}_{A},\Omega\right)  $ is the cyclic representation of $\left(
A,\mu\right)  $ and that $U\Omega=\Omega$, so $U$ is the unitary
representation of $\alpha$ in the cyclic representation. We can define the
\emph{finite orbit factor} $\mathbf{F}$ of $\mathbf{A}$ by defining $F$ to be
the von Neumann subalgebra of $A$ generated by the subgroup of elements of
$\Gamma$ which have finite orbits under $T$. Then we have the following fact.

\begin{proposition}
For a dual system $\mathbf{A}$ as described above, the finite orbit factor is
exactly $\mathbf{A}^{K}$.
\end{proposition}

\begin{proof}
Let $\mathbf{F}$ denote the finite orbit factor of $\mathbf{A}$, then we
simply have to prove that $F=A^{K}$. Since $\mathbf{F}$ is compact
\cite[Theorem 3.5]{D2}, it follows that $F\subset A^{K}$. So let's look at the
converse. The proof uses a similar idea as in the proof of \cite[Theorem
3.4]{D2}.

Let $l$ be the left regular representation of $\Gamma$ on $H$, i.e. $\left[
l(g)f\right]  (h)=f\left(  g^{-1}h\right)  $ for all $g,h\in\Gamma$ and $f\in
H$. So $A$ is generated by $\left\{  l(g):g\in\Gamma\right\}  $ and $F$ is
generated by $\left\{  l(g):g\in E\right\}  $, where $E:=\left\{  g\in
\Gamma:T^{\mathbb{Z}}(g)\text{ is finite}\right\}  $. Setting $H_{F}%
:=\overline{F\Omega}$, we see that since $l(g)\Omega=\delta_{g}$, the space
$H_{F}^{\perp}$ is spanned by the orthonormal set of vectors $\left\{
\delta_{g}:g\in G\backslash E\right\}  $. Since the elements of $G\backslash
E$ have infinite orbits under $T$, we have $\left\langle U^{n}\delta
_{g},\delta_{h}\right\rangle =0$ for $n$ large enough, for all $g,h\in
G\backslash E$. It follows that
\[
\lim_{n\rightarrow\infty}\left\langle U^{n}x,y\right\rangle =0
\]
for all $x,y\in H_{F}^{\perp}$. However, $H_{F}$ is invariant under $U$, so it
follows that%
\[
\lim_{n\rightarrow\infty}\left\langle U^{n}x,y\right\rangle =0
\]
for any $x\in H$ and $y\in H_{F}^{\perp}$.

Now, let $H_{0}$ be the subspace of $H$ spanned by eigenvectors of $U$. Since
$F\subset A^{K}$, it follows from \cite[Theorem 5.4]{D3} that $H_{F}%
\subset\overline{A^{K}\Omega}=H_{0}$. Suppose there is an eigenvector $v$ of
$U$ which is not in $H_{F}$, and let the corresponding eigenvalue be denoted
by $c$. (Since  $U$ is unitary, $|c|=1$.) Then there is a $y\in H_{F}^{\perp}$ such that $\left\langle
v,y\right\rangle \neq0$, hence the limit $\lim_{n\rightarrow\infty
}\left\langle U^{n}v,y\right\rangle =\left(  \lim_{n\rightarrow\infty}%
c^{n}\right)  \left\langle v,y\right\rangle $ is not zero (and does not even
necessarily exist), contradicting what we found above. It follows that such a
$v$ does not exist, therefore $H_{F}=H_{0}$.

Next we need to carry this result over to the algebras themselves. What we
have shown is that $\overline{F\Omega}=\overline{A^{K}\Omega}$, but $F$ is
invariant under the modular group of $\left(  A^{K},\mu^{K}\right)  $, since
$\mu^{K}=\mu|_{A^K}$ is tracial and its modular group therefore trivial. We have a
resulting conditional expectation $D:A^{K}\rightarrow F$ which is implemented
by the projection of $H_{0}$ onto $H_{F}$, which means that $D$ is the
identity mapping. So $F=A^{K}$.
\end{proof}

\begin{example}
Consider in particular the case where $\Gamma$ is the free group on any
infinite set of symbols $S$. Partition $S$ into two sets $S_{1}$ and $S_{2}$,
with at least $S_{2}$ infinite. Define the automorphism $T:\Gamma
\rightarrow\Gamma$ by extending any bijection $T:S\rightarrow S$ satisfying
$T(S_{1})=S_{1}$ and $T(S_{2})=S_{2}$, and with the orbits of $T$ on all
elements of $S_{1}$ being finite, but infinite on all elements of $S_{2}$. In
this way we obtain a dual system $\mathbf{A}$ as discussed above, and because
of Proposition 3.12, $\mathbf{A}^{K}$ is given by the von Neumann subalgebra
of $A$ generated by the subgroup of $\Gamma$ generated by $S_{1}$. In the same
way we also consider another dual system $\mathbf{B}$ with $\left(
B,\nu\right)  =\left(  A,\mu\right)  $, given by a bijection $K:S\rightarrow
S$, again satisfying $K(S_{1})=S_{1}$ and $K(S_{2})=S_{2}$, and such that
$K|_{S_{1}}=T|_{S_{1}}$ while all the orbits of $K$ on $S_{2}$ are finite.
Note that $\mathbf{B}$ is compact \cite[Theorem 3.5]{D2}. Furthermore,
$\mathbf{A}^{K}$ is a factor of the compact system $\mathbf{B}$, and it is
necessarily a modular factor, since $\mu$ is tracial and it modular group
therefore trivial. This example satisfies all the requirements of Corollary
3.9. It particular it provides examples where $\mathbf{A}^{K}$ is not trivial,
and not even an identity system, namely when $S_{1}$ has more than one element
and $T|_{S_{1}}$ is not the identity map. Note that the dynamics of $\mathbf{A}^{K}$
can in fact be fairly complicated. As an illustration, consider the countable 
set of symbols $S_1=\{s_1,s_2,s_3,...\}$, and let $T|_{S_1}$ be given by the 
cycles of increasing length $(s_1,s_2)$, $(s_3,s_4,s_5)$, $(s_6,s_7,s_8,s_9)$, 
and so on. So, for example, $T(s_6)=s_7$ and $T(s_9)=s_6$. Then there are elements
of $\overline{A^{K}\Omega}$ that do not have finite orbits under $U$, for example
$\sum_{n=1}^\infty \delta_{s_n}/n$, so we don't have periodic dynamics.
\end{example}

Next we exhibit a pair of disjoint compact systems by considering
$G=\mathbb{R}^{2}$ actions on quantum tori.

\begin{example}
Refer to Example 2.4, however consider $\theta_{1},\theta_{2} \in\mathbb{R}$.
Let $A:=M_{\theta_{1}}$ and $B:=M_{\theta_{2}}$, and let $\mu$ and $\nu$ be
their respective canonical traces. Furthermore, set
\[
\alpha_{s,t}:=\tau_{ps,qt}\text{ \ and \ }\beta_{s,t}:=\tau_{cs,dt}
\]
for all $\left(  s,t\right)  \in\mathbb{R}^{2}$, where $p,q,c,d\in
\mathbb{R}\backslash\left\{  0\right\}  $ are fixed. Then $\mathbf{A}=\left(
A,\alpha,\mu\right)  $ and $\mathbf{B}=\left(  B,\beta,\nu\right)  $ are
W*-dynamical systems.

Setting $\chi_{g,h}\left(  s,t\right)  :=e^{2\pi i\left(  gs+ht\right)  }$ for
all $\left(  g,h\right)  ,\left(  s,t\right)  \in\mathbb{R}^{2}$, one can use
harmonic analysis on $\mathbb{T}^{2}$ to show that $H$ is spanned by the
eigenvectors of the unitary group $U$ and of the unitary group $V$, namely
\[
U_{s,t}\chi_{m,n}=\chi_{mp,nq}(s,t)\chi_{m,n}
\]
and
\[
V_{s,t}\chi_{m,n}=\chi_{mc,nd}(s,t)\chi_{m,n}
\]
for all $(s,t)\in\mathbb{R}^{2}$ and $(m,n)\in\mathbb{Z}^{2}$. Therefore
$\mathbf{A}$ and $\mathbf{\tilde{B}}$ are compact W*-dynamical systems
\cite[Definition 2.5 and Proposition 2.6]{D2}, and their point spectra are
\[
\sigma_{\mathbf{A}}=\left\{  \chi_{mp,nq}:\left(  m,n\right)  \in
\mathbb{Z}^{2}\right\}
\]
and
\[
\sigma_{\mathbf{\tilde{B}}}=\left\{  \chi_{mc,nd}:\left(  m,n\right)
\in\mathbb{Z}^{2}\right\}
\]
respectively. Since $p,q,c,d\neq0$, one simultaneously sees that $\mathbf{A}$
and $\mathbf{\tilde{B}}$ are also ergodic, since the fixed point spaces of $U$
and $V$ are one dimensional (they are spanned by the eigenvector $\chi
_{0,0}=1=\Omega$). Now, if either $p/c$ or $q/d$, or both, are irrational, we
see that
\[
\sigma_{\mathbf{A}}\cap\sigma_{\mathbf{\tilde{B}}}=\left\{  1\right\}
\]
and therefore by \cite[Theorem 2.8]{D2} $\mathbf{A}$ and $\mathbf{\tilde{B}}$
are disjoint. A C*-algebraic version of this disjointness (though requiring
both $p/c$ and $q/d$ to be irrational) will be considered in Section 4.

When $\theta_{1}=\theta_{2}$, so $\left(  B,\nu\right)  =\left(  A,\mu\right)
$, one can use harmonic analysis on $\mathbb{T}^{2}$ and the Lebesgue
dominated convergence theorem, to show that for any $\ast$-automorphism
$\eta:A\rightarrow A$ with $\mu\circ\eta=\mu$ the function
\[
\mathbb{R}^{2}\rightarrow\mathbb{C}:\left(  s,t\right)  \mapsto\mu\left(
\eta\left(  \alpha_{s,t}(a)\right)  \beta_{s,t}(b)\right)
\]
is continuous for all $a,b\in A$, and therefore Theorem 3.8 applies.
\end{example}

We are not aware of pairs of disjoint weakly mixing W*- (or C*-) dynamical
systems on the same algebra and state.

\section{Unique ergodicity}

In this section we consider a simple connection between disjointness and
unique ergodicity. In particular we discuss examples of uniquely ergodic
C*-dynamical systems, and use this to obtain a C*-algebraic (rather than
W*-algebraic) example satisfying the assumptions of Corollary 3.4. We start
with the definition of unique ergodicity (also see \cite[Definition 4.5]{A1}).

\begin{definition}
A C*-dynamical system is \emph{uniquely ergodic} if it has a unique invariant state.
\end{definition}

Note that using arguments similar to the one in the proof of Theorem 3.3, one
can show that an amenable C*-dynamical system has at least one invariant state
by starting with an arbitrary state and considering its sequence of averages
over $\left(  \Lambda_{n}\right)  $, and furthermore that an amenable
C*-dynamical system is uniquely ergodic if and only if the limit
\[
\lim_{n\rightarrow\infty}\frac{1}{\left|  \Lambda_{n}\right|  }\int
_{\Lambda_{n}}\alpha_{g}(a)dg
\]
exists in the norm topology and is a scalar multiple of the identity for every
$a\in A$.

From properties of the Bochner integral (see \cite[Theorem II.2.6]{DU}) it
also follows that unique ergodicity implies ergodicity in the sense that in
the GNS representation obtained from the unique invariant state, the unitary
representation of the dynamics has a one dimensional fixed point space, or
equivalently
\[
\lim_{n\rightarrow\infty}\frac{1}{\left|  \Lambda_{n}\right|  }\int
_{\Lambda_{n}}\mu(\alpha_{g}(a)b)dg=\mu(a)\mu(b)
\]
for all $a,b\in A$, where on compact sets $g\mapsto\alpha_{g}(a)b$ is $\rho
$-measurable by \cite[Definition II.1.1]{DU}, and $g\mapsto\mu(\alpha
_{g}(a)b)$ is therefore $\rho$-measurable by Pettis's measurability theorem.
See \cite{dBDS} for the relevant general background regarding ergodicity.

The converse is not true however. For example, consider the shift $\ast
$-automorphism on a countably infinite tensor product $A$ of a C*-algebra $C$
with itself. Any state on $C$ then gives a shift invariant state on $A$,
leading to an ergodic (and in fact strongly mixing) state preserving dynamical
system. Since any state on $C$ will do, the shift is not uniquely ergodic though.

An example of unique ergodicity is provided by the quantum torus which
illustrates that one should indeed consider unique ergodicity for group
actions other than $\mathbb{Z}$:

\begin{proposition}
The C*-dynamical system $\left(  A_{\theta},\tau\right)  $ given by Example
2.4 is uniquely ergodic.
\end{proposition}

\begin{proof}
We know that the canonical trace Tr is invariant. Let $\mu$ be any invariant
state of $\left(  A_{\theta},\tau\right)  $. Let $B$ be the $\ast$-subalgebra
of $A_{\theta}$ generated by $\left\{  u,v\right\}  $, so $B$ consists of
finite linear combinations of elements of the form $u^{m}v^{n}$, with
$m,n\in\mathbb{Z}$. By (2.4) we have
\[
\mu\left(  u^{m}v^{n}\right)  =\mu\left(  \tau_{s,t}\left(  u^{m}v^{n}\right)
\right)  =e^{2\pi i\left(  ms+nt\right)  }\mu\left(  u^{m}v^{n}\right)
\]
which means that $m=n=0$ or $\mu\left(  u^{m}v^{n}\right)  =0$. However Tr
satisfies these conditions as well, so $\mu|_{B}=$ Tr$|_{B}$ and since $B$ is
dense in $A_{\theta}$, we have $\mu=$ Tr.
\end{proof}

This example is of course very simple in the sense that it is periodic in each
of the two real parameters of its action; in effect $\mathbb{R}^{2}$ acts on
$\mathbb{T}^{2}$ which in turn acts on $A_{\theta}$. A slightly more
complicated example is given by the following proposition.

\begin{proposition}
Let $\theta_{1},\theta_{2}\in\mathbb{R}$ and in terms of Example 2.4 set
$A:=A_{\theta_{1}}$ and $B:=A_{\theta_{2}}$. Furthermore, set $\alpha
_{s,t}:=\tau_{ps,qt}$ and $\beta_{s,t}:=\tau_{cs,dt}$ for all $\left(
s,t\right)  \in\mathbb{R}^{2}$, where $p,q,c,d\in\mathbb{R}\backslash\left\{
0\right\}  $ are fixed. Also assume that $p/c$ and $q/d$ are both irrational.
Then the C*-dynamical system $\left(  A\otimes_{m}B,\alpha\otimes_{m}%
\beta\right)  $ is uniquely ergodic.
\end{proposition}

\begin{proof}
Let $u$ and $v$ be the generators of $A$ as in Example 2.4, and let $w$ and
$z$ correspondingly be the generators of $B$. Using the notation $\dot
{u}=u\otimes1$, $\dot{v}=v\otimes1$, $\dot{w}=1\otimes w$ and
$\dot{z}=1\otimes z$, we see that the C*-algebra $A\otimes_{m}B$ is generated
by $\left\{  \dot{u},\dot{v},\dot{w},\dot{z}\right\}  $. Therefore the $\ast
$-algebra $C$ consisting of all finite linear combinations of elements of the
form $\dot{u}^{j}\dot{v}^{k}\dot{w}^{l}\dot{z}^{m}$, where $j,k,l,m\in
\mathbb{Z}$, is dense in $A\otimes_{m}B$. If $\omega$ is any invariant state
of $\left(  A\otimes_{m}B,\alpha\otimes_{m}\beta\right)  $ then we use a
similar argument as in Proposition 4.2:
\[
\omega\left(  \dot{u}^{j}\dot{v}^{k}\dot{w}^{l}\dot{z}^{m}\right)  =e^{2\pi
i\left[  \left(  jp+lc\right)  s+\left(  kq+md\right)  t\right]  }
\omega\left(  \dot{u}^{j}\dot{v}^{k}\dot{w}^{l}\dot{z}^{m}\right)
\]
for all $\left(  s,t\right)  \in\mathbb{R}^{2}$, therefore $\omega\left(
\dot{u}^{j}\dot{v}^{k}\dot{w}^{l}\dot{z}^{m}\right)  =0$ or $jp+lc=kq+md=0$.
Because of the irrationality of $p/c$ and $q/d$, the latter implies
$j=k=l=m=0$, and in this case $\omega\left(  \dot{u}^{j}\dot{v}^{k}\dot{w}%
^{l}\dot{z}^{m}\right)  =\omega\left(  1\otimes1\right)  =1=$ Tr$\otimes_{m}%
$Tr$\left(  \dot{u} ^{j}\dot{v}^{k}\dot{w}^{l}\dot{z}^{m}\right)  $. When at
least one of $j$, $k$, $l$ or $m$ is not zero we have Tr$\otimes_{m}%
$Tr$\left(  \dot{u}^{j} \dot{v}^{k}\dot{w}^{l}\dot{z}^{m}\right)  =$
Tr$\left(  u^{j}v^{k}\right)  $Tr$\left(  w^{l}z^{m}\right)  =0$. This proves
that $\omega=$ Tr$\otimes_{m}$Tr.
\end{proof}

Note that if $ps,cs\in\mathbb{Z}\backslash\left\{  0\right\}  $ then $p/c$ is
rational, and similarly for $q/d$. Therefore the product system in Proposition
4.3 is not periodic in either of the two real parameters; consider for example
the orbit of $\dot{u}+\dot{w}$.

\begin{remark}
In general, if we have two C*-dynamical systems $\left(  A,\mu\right)  $ and
$\left(  B,\nu\right)  $ with $G$ a topological group and $g\mapsto\alpha
_{g}(a)$ and $g\mapsto\beta_{g}(b)$ both continuous, then it is fairly
straightforward to show that $g\mapsto\alpha_{g}\otimes_{m}\beta_{g}(c)$ is
continuous for all $c\in A\otimes_{m}B$. Hence, if furthermore $\left(
A,\mu\right)  $ and $\left(  B,\nu\right)  $ are also amenable, so is their
product system. In particular the dynamics of the product system in
Proposition 4.3 is continuous in the sense just described, and the product
system is therefore amenable.
\end{remark}

Clearly Proposition 4.3 is similar to the disjointness result in Example 3.14,
although now we are in the C*-algebra context. Let us formalize this
connection: Two uniquely ergodic C*-dynamical systems $\left(  A,\alpha
\right)  $ and $\left(  B,\beta\right)  $ are called \emph{disjoint} if the
corresponding state preserving C*-dynamical systems are disjoint. Clearly this
is the case if and only if the product system $\left(  A\otimes_{m}%
B,\alpha\otimes_{m}\beta\right)  $ is uniquely ergodic, since any invariant
state of this product system is a joining because it restricts to invariant
states of $\left(  A,\alpha\right)  $ and $\left(  B,\beta\right)  $, which
are unique. In particular we have the following corollary which tells us that
examples of disjointness are not restricted to W*-dynamical systems:

\begin{corollary}
The C*-dynamical systems $\left(  A,\alpha\right)  $ and $\left(
B,\beta\right)  $ described in Proposition 4.3 are disjoint.
\end{corollary}

In Section 3 we considered consequences of Theorem 3.3 for W*-dynamical
systems. We now consider an example satisfying all the requirements of Theorem
3.3 (specifically those of Corollary 3.4) in the C*-algebraic context.

\begin{example}
Still consider the situation in Proposition 4.3, and let $\mu$ and $\nu$ be
the canonical traces of $A$ and $B$ respectively. Let $u,v$ and $w,z$ be the
generators of $A$ and $B$ respectively, as in the proof of Proposition 4.3. It
is easily verified that $vz=zv$, so the C*-subalgebras $C^{\ast}(v)$ and
$C^{\ast}(z)$ of $B(H)$ generated by $v$ and $z$ respectively, are mutually
commuting. We can therefore use a similar idea as in Section 3 to construct a
coupling of $\left(  A,\mu\right)  $ and $\left(  B,\nu\right)  $. Let
\[
\delta:C^{\ast}(v)\odot C^{\ast}(z)\rightarrow B(H)
\]
be the $\ast$-homomorphism obtained by extending $\delta(a\otimes b)=ab$. We
can construct a conditional expectation $D_{1}:A\rightarrow C^{\ast}(v)$ such
that $D_{1}(u^{m}v^{n})=0$ for nonzero $m\in\mathbb{Z}$, and $D_{1}%
(v^{n})=v^{n}$ for all $n\in\mathbb{Z}$. Since $\mu(u^{m}v^{n})$ is $1$ when
$m=n=0$, and $0$ otherwise, it follows that $\mu\circ D_{1}=\mu$. Similarly we
obtain $D_{2}:B\rightarrow C^{\ast}(z)$ with $\nu\circ D_{2}=\nu$. We then
define a state $\kappa$ on $A\otimes_{m}B$ by continuously extending the
following:%
\[
\kappa\left(  c\right)  :=\left\langle \Omega,\delta\circ\left(  D_{1}\odot
D_{2}\right)  (c)\Omega\right\rangle
\]
for all $c\in A\odot B$. It is now easy to verify that $\kappa$ is a coupling
of $\left(  A,\mu\right)  $ and $\left(  B,\nu\right)  $. It is indeed a
nontrivial coupling, that is, $\kappa\neq\mu\otimes_{m}\nu$, since
$\kappa(v\otimes z^{-1})=1$ while $\mu\otimes_{m}\nu\left(  v\otimes
z^{-1}\right)  =0$. We therefore have all the ingredients required in Theorem
3.3, or more specifically, Corollary 3.4. Note that as opposed to the
W*-dynamical applications considered in Section 3, this gives a case of
Theorem 3.3 for two different algebras. Explicitly, Corollary 3.4 and
Corollary 4.5 say that
\[
\lim_{n\rightarrow\infty}\frac{1}{\left|\Lambda_{n}\right|}
\int_{\Lambda_{n}}\kappa\left(\alpha_{g}(a)\otimes\beta_{g}(b)\right)dg
=\mu(a)\nu(b)
\]
for all $a\in A$ and $b\in B$, where $\left(  \Lambda_{n}\right)  $ is any F\o
lner sequence in $G=\mathbb{R}^{2}$.

In particular, when $\theta_{1}=\theta_{2}$ (so $A=B$), and if we write
$D=D_{1}=D_{2}$, we have%
\[
\lim_{n\rightarrow\infty}\frac{1}{\left|  \Lambda_{n}\right|  }\int
_{\Lambda_{n}}\mu\left(  D\left(  \alpha_{g}(a)\right)  D\left(  \beta
_{g}(b)\right)  \right)  dg=\mu(a)\mu(b)
\]
for all $a,b\in A$.
\end{example}

\section{Relative unique ergodicity}

In this section we consider a relative version of unique ergodicity and its
connections with disjointness. We first prove \cite[Theorem 3.2]{AD} for
actions of more general groups than $\mathbb{Z}$. It should be noted that
\cite[Theorem 3.2]{AD} has already lead to further work, in particular
\cite{FM} and \cite{AM}, but only for actions of the group $\mathbb{Z}$. In
this paper our general setup allow for, and some of our examples require, more
general groups. Indeed, after proving the result for more general groups, we
again discuss examples using quantum tori, for actions of the group
$\mathbb{R}$. Lastly we use these results to obtain an example of (relative)
disjointness in the C*-algebraic context, and use it to illustrate Theorem
3.3. 

\begin{definition}
We call the C*-dynamical system $\left(  A,\alpha\right)  $ \emph{uniquely
ergodic relative to its fixed point algebra} if every state on $A^{\alpha}$
has a unique extension to an invariant state of $\left(  A,\alpha\right)  $.
\end{definition}

Note that from the remarks following Definition 4.1, the fixed point algebra
of a uniquely ergodic amenable C*-dynamical system is $A^{\alpha}=\mathbb{C}1$.

To prove the next result, we follow the basic plan of \cite{AD}'s proof. Note
that if we say that a conditional expectation $E:A\rightarrow A^{\alpha}$ is
$\alpha$\emph{-invariant}, we mean that $E\circ\alpha_{g}=E$ for all $g\in G$,
and similarly for linear functionals. Existence of limits, closures etc. are
all in terms of the norm topology on $A$.

\begin{theorem}
Let $\left(  A,\alpha\right)  $ be an amenable C*-dynamical system with $G$
unimodular, and let $\left(  \Lambda_{n}\right)  $ be both a right and left
F\o lner sequence. Then statements (i) to (v) below are equivalent.

(i) The system $\left(  A,\alpha\right)  $ is uniquely ergodic relative to its
fixed point algebra.

(ii) The limit
\[
\lim_{n\rightarrow\infty}\frac{1}{\left|  \Lambda_{n}\right|  }\int
_{\Lambda_{n}}\alpha_{g}(a)dg
\]
exists for every $a\in A$.

(iii) The subspace $A^{\alpha}+$ span$\left\{  a-\alpha_{g}(a):g\in G,a\in
A\right\}  $ is dense in $A$.

(iv) The equality $A=A^{\alpha}+\overline{\text{span}\left\{  a-\alpha
_{g}(a):g\in G,a\in A\right\}  }$ holds.

(v) Every bounded linear functional on $A^{\alpha}$ has a unique bounded
$\alpha$-invariant extension to $A$ with the same norm.

Furthermore, statements (i) to (v) imply the following statements:

(vi) There exists a unique $\alpha$-invariant conditional expectation $E$ from
$A$ onto $A^{\alpha}$.

(vii) The conditional expectation $E$ in (vi) is given by
\[
E(a)=\lim_{n\rightarrow\infty}\frac{1}{\left|  \Lambda_{n}\right|  }
\int_{\Lambda_{n}}\alpha_{g}(a)dg
\]
for all $a\in A$.
\end{theorem}

\begin{proof}
We first discuss the implications that need the fewest assumptions. In the
first four implications below, we don't need amenability of $(A,\alpha)$, and
$G$ can be arbitrary.

(v)$\Rightarrow$(i): Let $\omega$ be any state on $A^{\alpha}$. Then there is
an $\alpha$-invariant linear functional $\mu$ on $A$, extending $\omega$, such
that $\left\|  \mu\right\|  =\left\|  \omega\right\|  =\omega(1)=\mu(1)$,
since $1\in A^{\alpha}$, therefore $\mu$ is positive and hence a state.

(i)$\Rightarrow$(iii): Suppose not, so (i) holds but there exists a $b\in
A\backslash Z$ where
\[
Z:=\overline{A^{\alpha}+\text{span}\left\{  a-\alpha_{g}(a):g\in G,a\in
A\right\}  }\text{.}
\]
Then by the Hahn-Banach theorem there exists a bounded linear functional
$f:A\rightarrow\mathbb{C}$ such that $f(b)\neq0$ while $f|_{Z}=0$. In
particular this tells us that $f$ is $\alpha$-invariant and $f|_{A^{\alpha}
}=0$.

Now consider the linear functionals $f_{1}$ and $f_{2}$ on $A$ defined by
$f_{1}(a):=f(a)+\overline{f(a^{\ast})}$ and $f_{2}(a):=[f(a)-\overline
{f(a^{\ast})}]/i$. Then either $f_{1}(b)\neq0$ or $f_{2}(b)\neq0$, since
$f_{1}+if_{2}=2f$. Furthermore $f_{1}|_{Z}=f_{2}|_{Z}=0$, since $Z$ is closed
under involution. So $f_{1}$ and $f_{2}$ are both $\alpha$-invariant and
$f_{1}|_{A^{\alpha}}=f_{2}|_{A^{\alpha}}=0$. Note that $f_{1}$ and $f_{2}$ are
both self-adjoint, meaning $\overline{f_{j}(a^{\ast})}=f_{j}(a)$, therefore we
have shown that without loss we can assume $f$ to be self-adjoint by replacing
it by $f_{1}$ or $f_{2}$, whichever is not zero at $b$.

Let $f=f_{+}-f_{-}$ be the Jordan decomposition of $f$, so $f_{+}$ and $f_{-}$
are both positive and $\left\|  f\right\|  =\left\|  f_{+}\right\|  +\left\|
f_{-}\right\|  $. Since $f_{+}$ and $f_{-}$ are unique, $\left\|  f_{\pm}%
\circ\alpha_{g}\right\|  =\left\|  f_{\pm}\right\|  $, and $\alpha_{g}$ is
positive, it follows from the $\alpha$-invariance of $f$ that both $f_{+}$ and
$f_{-}$ are $\alpha$-invariant. Since $f|_{Z}=0$, we have in particular that
$f_{+}|_{A^{\alpha}}=f_{-}|_{A^{\alpha}}$. One possibility is $f_{+}%
|_{A^{\alpha}}=f_{-}|_{A^{\alpha}}=0$, in which case $\left\|  f_{\pm
}\right\|  =f_{\pm}(1)=0$ by positivity, contradicting $f(b)\neq0$. The
remaining possibility is $f_{+}|_{A^{\alpha}}=f_{-}|_{A^{\alpha}}\neq0$, in
which case $\left\|  f_{\pm}\right\|  =f_{\pm}(1)=f_{\pm}|_{A^{\alpha}%
}(1)=\left\|  f_{\pm}|_{A^{\alpha}}\right\|  $ allowing us to normalize both
$f_{+}$ and $f_{-}$ with the same factor (namely $1/\left\|  f_{+}%
|_{A^{\alpha}}\right\|  =1/\left\|  f_{-}|_{A^{\alpha}}\right\|  $) to obtain
states $\omega_{+}$ and $\omega_{-}$ with $\omega_{+}|_{A^{\alpha}}=\omega
_{-}|_{A^{\alpha}}$. By (i) the latter means that $\omega_{+}=\omega_{-}$ and
therefore $f_{+}=f_{-}$ which gives $f=0$, again contradicting $f(b)\neq0$. We
conclude that no such $b$ exists.

(v) and (vi)$\Rightarrow$(iv): Take any $a\in A$, then $E(a)\in A^{\alpha}$
and $a-E(a)\in\ker E$, proving that $A=A^{\alpha}+\ker E$. To show (iv) we
therefore only need to obtain a contradiction if we suppose there exists a
$b\in\ker E\backslash Z$ where $Z:=\overline{\text{span}\left\{  a-\alpha
_{g}(a):g\in G,a\in A\right\}  }$. Applying the Hahn-Banach theorem, the
existence of such a $b$ means that there is a bounded linear functional $f$ on
$A$ such that $f(b)\neq0$ while $f|_{Z}=0$. In particular the latter says that
$f$ is $\alpha$-invariant. Using the same argument as above we can assume
without loss that $f$ is positive. Now, $f$ is a linear extension of
$f|_{A^{\alpha}}$ and by positivity they have the same norm, since $\left\|
f\right\|  =f(1)=f|_{A^{\alpha}}(1)=\left\|  f|_{A^{\alpha}}\right\|  $.
However $f|_{A^{\alpha}}\circ E=$ $f\circ E$ is also such an extension of
$f|_{A^{\alpha}}$, therefore by the uniqueness in (v) we have $f\circ E=f$.
This means that $f(b)=f(E(b))=0$, since $b\in\ker E$, contradicting
$f(b)\neq0$.

(iv)$\Rightarrow$(iii): This is immediate.

(iii)$\Rightarrow$(ii): Now we use the fact that $\left(  A,\alpha\right)  $
is amenable and $\left(  \Lambda_{n}\right)  $ is a right F\o lner sequence.
For any $a\in A^{\alpha}$ the limit in (ii) is just $a$ by simple properties
of the Bochner integral. Secondly, for an arbitrary $a\in A$ and $h\in G$ we
have, using properties of the Bochner integral, that%
\begin{align*}
&  \left\|  \frac{1}{\left|  \Lambda_{n}\right|  }\int_{\Lambda_{n}}\alpha
_{g}\left(  a-\alpha_{h}(a)\right)  dg\right\| \\
&  =\left\|  \frac{1}{\left|  \Lambda_{n}\right|  }\int_{\Lambda_{n}}
\alpha_{g}\left(  a\right)  dg-\frac{1}{\left|  \Lambda_{n}\right|  }
\int_{\Lambda_{n}h}\alpha_{g}\left(  a\right)  dg\right\| \\
&  \leq\left\|  \frac{1}{\left|  \Lambda_{n}\right|  }\int_{\Lambda
_{n}\backslash\left(  \Lambda_{n}\cap\left(  \Lambda_{n}h\right)  \right)
}\alpha_{g}\left(  a\right)  dg\right\|  +\left\|  \frac{1}{\left|
\Lambda_{n}\right|  }\int_{\left(  \Lambda_{n}h\right)  \backslash\left(
\Lambda_{n}\cap\left(  \Lambda_{n}h\right)  \right)  }\alpha_{g}\left(
a\right)  dg\right\| \\
&  \leq\frac{1}{\left|  \Lambda_{n}\right|  }\left|  \Lambda_{n}
\bigtriangleup\left(  \Lambda_{n}h\right)  \right|  \left\|  a\right\|
\end{align*}
which tends to zero as $n\rightarrow\infty$. Combining these two facts, we
have shown that the limit in (ii) converges for all $a$ in $A^{\alpha} +$
span$\left\{  a-\alpha_{g}(a):g\in G,a\in A\right\}  $, and since the latter
is dense in $A$ a simple argument shows that (ii) holds for all $a\in A$.

(ii)$\Rightarrow$(v), (vi) and (vii): From now on we need all the assumptions
in the theorem. Since the limit in (ii) exists, it gives a well-defined
function $E:A\rightarrow A$ by the formula in (vii), which is clearly linear.
Using a standard property of the Bochner integral \cite[Theorem II.2.4(ii)]%
{DU} we have
\[
\left\|  E(a)\right\|  \leq\lim_{n\rightarrow\infty}\frac{1}{\left|
\Lambda_{n}\right|  }\int_{\Lambda_{n}}\left\|  a\right\|  dg=\left\|
a\right\|
\]
for all $a\in A$, so $\left\|  E\right\|  \leq1$. Since $E(1)=1$, we in fact
have $\left\|  E\right\|  =1$. Note that
\[
E(\alpha_{h}(a))=\lim_{n\rightarrow\infty}\frac{1}{\left|  \Lambda_{n}\right|
}\int_{\Lambda_{n}h}\alpha_{g}(a)dg
\]
but as above
\[
\lim_{n\rightarrow\infty}\left(  \frac{1}{\left|  \Lambda_{n}\right|  }
\int_{\Lambda_{n}}\alpha_{g}(a)dg-\frac{1}{\left|  \Lambda_{n}\right|  }
\int_{\Lambda_{n}h}\alpha_{g}(a)dg\right)  =0
\]
which means that $E(\alpha_{h}(a))=E(a)$, i.e. $E$ is $\alpha$-invariant. From
the definition of $E$ as the limit in (ii) we have $A^{\alpha}\subset E(A)$.
But using the properties of Bochner integrals (see in particular \cite[Theorem
II.2.6]{DU}) and the fact that $G$ is unimodular and $\left(  \Lambda
_{n}\right)  $ is also a left F\o lner sequence, we have
\begin{align*}
\alpha_{h}(E(a))  &  =\lim_{n\rightarrow\infty}\frac{1}{\left|  \Lambda
_{n}\right|  }\int_{\Lambda_{n}}\alpha_{hg}(a)dg\\
&  =\lim_{n\rightarrow\infty}\frac{1}{\left|  \Lambda_{n}\right|  }
\int_{h\Lambda_{n}}\alpha_{g}(a)dg\\
&  =\lim_{n\rightarrow\infty}\frac{1}{\left|  \Lambda_{n}\right|  }
\int_{\Lambda_{n}}\alpha_{g}(a)dg\\
&  =E(a)
\end{align*}
noting that $g\mapsto\alpha_{hg}(a)$ is indeed $\rho$-measurable on compact
subsets of $G$ by Pettis's measurability theorem since $g\mapsto\alpha_{g}(a)$
is, and for any bounded linear functional $f$ on $A$ one has $f(\alpha
_{hg}(a))=(f\circ\alpha_{h})(\alpha_{g}(a))$ with $f\circ\alpha_{h}$ a bounded
linear functional on $A$. This implies that $E(A)\subset A^{\alpha}$. It
follows that $E$ is a norm 1 projection of $A$ onto $A^{\alpha}$. Furthermore
\[
A^{\alpha}=\bigcap_{g\in G}\left(  \alpha_{g}-\text{ id}_{A}\right)  ^{-1}(0)
\]
so $A^{\alpha}$ is closed and therefore a C*-subalgebra of $A$. Therefore $E$
is a conditional expectation by a result of Tomiyama \cite{T} (also see
\cite[Section 9.1]{St}).

Supposing $E^{\prime}$ is also a $\alpha$-invariant conditional expectation of
$A$ onto $A^{\alpha}$, we can again use the basic properties of Bochner
integrals to obtain
\begin{align*}
E^{\prime}(a)  &  =\lim_{n\rightarrow\infty}\frac{1}{\left|  \Lambda
_{n}\right|  }\int_{\Lambda_{n}}E^{\prime}(a)dg\\
&  =\lim_{n\rightarrow\infty}\frac{1}{\left|  \Lambda_{n}\right|  }
\int_{\Lambda_{n}}E^{\prime}\left(  \alpha_{g}(a)\right)  dg\\
&  =E^{\prime}\left(  \lim_{n\rightarrow\infty}\frac{1}{\left|  \Lambda
_{n}\right|  }\int_{\Lambda_{n}}\alpha_{g}(a)dg\right) \\
&  =E^{\prime}\left(  E(a)\right) \\
&  =E(a)
\end{align*}
for every $a\in A$, since $E(a)\in A^{\alpha}$, meaning $E$ is indeed unique.

It remains to show (v). Let $f$ be any bounded linear functional on
$A^{\alpha}$. Then $f\circ E$ is a bounded linear functional on $A$ with the
same norm, and it is $\alpha$-invariant since $E$ is. Supposing $f^{\prime}$
is also an $\alpha$-invariant bounded linear functional on $A$ extending $f$
and with the same norm as $f$, we have
\begin{align*}
f^{\prime}(a)  &  =\lim_{n\rightarrow\infty}\frac{1}{\left|  \Lambda
_{n}\right|  }\int_{\Lambda_{n}}f^{\prime}(a)dg\\
&  =\lim_{n\rightarrow\infty}\frac{1}{\left|  \Lambda_{n}\right|  }
\int_{\Lambda_{n}}f^{\prime}\left(  \alpha_{g}(a)\right)  dg\\
&  =f^{\prime}\left(  \lim_{n\rightarrow\infty}\frac{1}{\left|  \Lambda
_{n}\right|  }\int_{\Lambda_{n}}\alpha_{g}(a)dg\right) \\
&  =f^{\prime}\left(  E(a)\right) \\
&  =f\left(  E(a)\right)
\end{align*}
for all $a\in A$, showing that $f^{\prime}=f\circ E$. This proves the
uniqueness in (v).
\end{proof}

\begin{remark}
From the proof it is clear that various implications hold with less
assumptions, for example (i)$\Rightarrow$(iii) doesn't require $\left(
A,\alpha\right)  $ to amenable, and in particular $G$ can be arbitrary. Also
note that if $G$ is abelian, or if more generally we restrict ourselves to F\o
lner sequences that are both right and left F\o lner sequences, then the value
of the ergodic average in (ii) is independent of the F\o lner sequence because
of (vi) and (vii).
\end{remark}

The equivalence of conditions (i) and (ii) is what we are most interested in
now, since it is useful in showing relative unique ergodicity in the examples
below. These examples are simple variations on Propositions 4.2 and 4.3.

\begin{proposition}
Consider the situation in Example 2.4, write $A=A_{\theta}$ and set
$\alpha_{s}:=\tau_{s,0}$ for all $s\in\mathbb{R}$. Then $\left(
A,\alpha\right)  $ is an example of a C*-dynamical system for an action of
$\mathbb{R}$ which is uniquely ergodic relative to its fixed point algebra.
\end{proposition}

\begin{proof}
Take any F\o lner sequence $\left(  \Lambda_{n}\right)  $ in $\mathbb{R}$, for
example $\Lambda_{n}=\left[  0,n\right]  $. Let $a=u^{j}v^{k}$ for any
$j,k\in\mathbb{Z}$. Then
\begin{align*}
\lim_{n\rightarrow\infty}\frac{1}{\left|  \Lambda_{n}\right|  }\int
_{\Lambda_{n}}\alpha_{s}(a)ds  &  =\lim_{n\rightarrow\infty}\frac{1}{\left|
\Lambda_{n}\right|  }\int_{\Lambda_{n}}e^{2\pi ijs}ads\\
&  =a\lim_{n\rightarrow\infty}\frac{1}{\left|  \Lambda_{n}\right|  }
\int_{\Lambda_{n}}e^{2\pi ijs}ds\\
&  =\left\{
\begin{array}
[c]{l}%
a\text{ if }j=0\\
0\text{ otherwise}%
\end{array}
\right.
\end{align*}
and therefore $\lim_{n\rightarrow\infty}\frac{1}{\left|  \Lambda_{n}\right|
}\int_{\Lambda_{n}}\alpha_{s}(a)ds$ exists for all $a$ in the dense $\ast
$-subalgebra of $A$ generated by $\left\{  u,v\right\}  $. Then it is easily
verified that it exists for all $a\in A$. According to Theorem 5.2(i) and (ii)
this means that $\left(  A,\alpha\right)  $ is indeed uniquely ergodic
relative to its fixed point algebra.
\end{proof}

\begin{remark}
With some more work one can show that the fixed point algebra of the
C*-dynamical system in Proposition 5.4 is the C*-subalgebra of $A$ generated
by $v$, as one would expect. One way of doing this is to use the conditional
expectation $D_1$ (for $\theta_1=\theta$) mentioned in Example 4.6; 
this is exactly the conditional expectation also given by Theorem 5.2(vi) 
for the system in Proposition 5.4.
\end{remark}

\begin{proposition}
Consider the situation in Proposition 4.3 except that $p/c$ and $q/d$ need not
be irrational. Write $\alpha_{s}\equiv\alpha_{s,0}$ and $\beta_{s}\equiv
\beta_{s,0}$ for all $s\in\mathbb{R}$, to define actions $\alpha$ and $\beta$
of $\mathbb{R}$ on $A$ and $B$ respectively. Then $\left(  A\otimes
_{m}B,\alpha\otimes_{m}\beta\right)  $ is an example of a C*-dynamical system
for an action of $\mathbb{R}$ which is uniquely ergodic relative to its fixed
point algebra.
\end{proposition}

\begin{proof}
Using the notation in the proof of Proposition 4.3 and setting $a=\dot{u}%
^{j}\dot{v}^{k}\dot{w}^{l}\dot{z}^{m}$ for any $j,k,l,m\in\mathbb{Z}$, we
have
\begin{align*}
\lim_{n\rightarrow\infty}\frac{1}{\left|  \Lambda_{n}\right|  }\int
_{\Lambda_{n}}\left(  \alpha\otimes_{m}\beta\right)  _{s}(a)ds  &
=\lim_{n\rightarrow\infty}\frac{1}{\left|  \Lambda_{n}\right|  }\int
_{\Lambda_{n}}e^{2\pi i\left(  jp+lc\right)  s}ads\\
&  =\left\{
\begin{array}
[c]{l}%
a\text{ if }jp+lc=0\\
0\text{ otherwise}%
\end{array}
\right.
\end{align*}
and proceeding as in Proposition 5.4's proof, the result follows.
\end{proof}

As in Section 4, in the case where $p/c$ is irrational, the C*-dynamical
system $\left(  A\otimes_{m}B,\alpha\otimes_{m}\beta\right)  $ in Proposition
5.6 is not periodic.

Using higher dimensional quantum tori one should similarly be able to
construct C*-dynamical systems for actions of $\mathbb{R}^{n}$ which are
uniquely ergodic relative to their fixed point algebras.

Finally we return to disjointness. In Section 3 we focussed on state
preserving C*-dynamical systems, that is to say, we considered a specific
invariant state. Now however we in principle allow many states, so let us
consider an appropriate version of disjointness for relatively uniquely
ergodic systems. In line with relative unique ergodicity, the focus here will
be on fixed point algebras, rather than more general factors as in the
W*-dynamical case in Section 3.

Note that if $\left(  A,\alpha\right)  $ and $\left(  B,\beta\right)  $ are
C*-dynamical systems which are uniquely ergodic relative to their respective
fixed point algebras, then any state $\psi$ on the algebraic tensor product
$A^{\alpha}\odot B^{\beta}$ induces a unique pair of invariant states $\mu$
and $\nu$ of $\left(  A,\alpha\right)  $ and $\left(  B,\beta\right)  $
respectively through the following process: $\psi|_{A^{\alpha}\odot1}$ gives a
state on $A^{\alpha}$, which has a unique extension to an invariant state
$\mu$ of $\left(  A,\alpha\right)  $, and similarly for $\nu$.

\begin{definition}
Let the C*-dynamical systems $\left(  A,\alpha\right)  $ and $\left(
B,\beta\right)  $ be uniquely ergodic relative to their respective fixed point
algebras. Let $R$ be the closure of $A^{\alpha}\odot B^{\beta}$ in
$A\otimes_{m}B$. We call $\left(  A,\alpha\right)  $ and $\left(
B,\beta\right)  $ \emph{disjoint relative to their fixed point algebras} if
every state $\psi$ on $R$ has a unique extension to a joining of $\left(
A,\alpha,\mu\right)  $ and $\left(  B,\beta,\nu\right)  $, where $\mu$ and
$\nu$ are the states induced by $\psi$ on $A$ and $B$ respectively.
\end{definition}

Note that the disjointness defined in Definition 5.7, implies that for every
state $\psi$ on $R$, the induced state preserving C*-dynamical systems
$\left(  A,\alpha,\mu\right)  $ and $\left(  B,\beta,\nu\right)  $ are
disjoint with respect to $\left(  R,\text{id},\psi\right)  $, in the sense of
Definition 3.1.

As in Section 4, the connection with relative unique ergodicity we want to
exploit, is quite simple: Let the C*-dynamical systems $\left(  A,\alpha
\right)  $ and $\left(  B,\beta\right)  $ be uniquely ergodic relative to
their respective fixed point algebras. Keep in mind that since $A^{\alpha
}\odot B^{\beta}$ is in the fixed point algebra of $\left(  A\otimes
_{m}B,\alpha\otimes_{m}\beta\right)  $, so is $R$. If $\left(  A\otimes
_{m}B,\alpha\otimes_{m}\beta\right)  $ is uniquely ergodic relative to $R$ (so
in particular $R$ is equal to the fixed point algebra of $\left(  A\otimes
_{m}B,\alpha\otimes_{m}\beta\right)  $), then $\left(  A,\alpha\right)  $ and
$\left(  B,\beta\right)  $ are disjoint relative to their fixed point
algebras. This is the case, since any state $\psi$ on $R$ has a unique
extension to an invariant state of $\left(  A\otimes_{m}B,\alpha\otimes
_{m}\beta\right)  $, which by the relative unique ergodicity of $\left(
A,\alpha\right)  $ and $\left(  B,\beta\right)  $ is then necessarily the
unique joining of $\left(  A,\alpha,\mu\right)  $ and $\left(  B,\beta
,\nu\right)  $ as in Definition 5.7. In general though, $R$ will not be the
fixed point algebra, so this is something we need to check.

\begin{corollary}
Consider the situation in Proposition 5.6, and assume furthermore that $p/c$
is irrational, and that $\theta_{1}$ or $\theta_{2}$ are irrational. Then
$\left(  A,\alpha\right)  $ and $\left(  B,\beta\right)  $ are disjoint
relative to their fixed point algebras.
\end{corollary}

\begin{proof}
The irrationality of $\theta_{1}$ or $\theta_{2}$ ensures that either $A$ or
$B$ are nuclear, so $A\otimes_{m}B=A\otimes B$ is the spatial tensor product,
which allows us to use the spatial representation as will be seen below. The
key technical point is to show that the fixed point algebra of $\left(
A\otimes B,\alpha\otimes\beta\right)  $, is indeed $A^{\alpha}\otimes
B^{\beta}$. We again use the notation in the proof of Proposition 4.3, and
$H=L^{2}\left(  \mathbb{T}^{2}\right)  $.

Denote the unitary group representing $\alpha\otimes\beta$ on $H\otimes H$ by
$V$. So $V_{s}f=f\circ\left(  T_{ps,0}\times T_{cs,0}\right)  $ in terms of
Example 2.4. Using harmonic analysis on $\mathbb{T}^{2}$ and the fact that
$p/c$ is irrational, one can show that the fixed point space $\left(  H\otimes
H\right)  ^{V}$ of $V$ is the Hilbert subspace of $H\otimes H$ spanned by
$\left\{  e_{0,k}\otimes e_{0,m}:k,m\in\mathbb{Z}\right\}  $, where $e_{j,k}$
denote the functions in the total orthonormal set of $H$ given by
$e_{j,k}\left(  s,t\right)  :=e^{2\pi ijs}e^{2\pi ikt}$.

Furthermore, taking the spatial tensor product of the conditional expectations
$D_{1}$ and $D_{2}$ described in Example 4.6, we obtain the conditional
expectation $E:A\otimes B\rightarrow C^{\ast}(v)\otimes C^{\ast}(z)$. One can
then show that $\left[  (1-E)\left(  A\otimes B\right)  \right]
(\Omega\otimes\Omega)$ is orthogonal to $\left(  H\otimes H\right)  ^{V}$.
Since $\Omega\otimes\Omega$ is separating for $A\otimes B$ on $H\otimes H$, it
follows that the only fixed point of $(1-E)\left(  A\otimes B\right)  $ under
$\alpha\otimes\beta$ is $0$. For any $a\in\left(  A\otimes B\right)
^{\alpha\otimes\beta}$ we then have $a=\left(  \alpha\otimes\beta\right)
_{g}(a)=\left(  \alpha\otimes\beta\right)  _{g}(Ea+(1-E)a)=Ea+\left(
\alpha\otimes\beta\right)  _{g}((1-E)a)$, since $C^{\ast}(v)\otimes C^{\ast
}(z)\subset\left(  A\otimes B\right)  ^{\alpha\otimes\beta}$. So $(1-E)a$ is a
fixed point of $\alpha\otimes\beta$, which means that $(1-E)a=0$. Therefore
$a=Ea\in C^{\ast}(v)\otimes C^{\ast}(z)$, so $\left(  A\otimes B\right)
^{\alpha\otimes\beta}=C^{\ast}(v)\otimes C^{\ast}(z)$.

As mentioned above, $C^{\ast}(v)$ and $C^{\ast}(z)$ are the respective fixed
algebras of $\left(  A,\alpha\right)  $ and $\left(  B,\beta\right)  $. By
Proposition 5.6 we now know that $\left(  A,\alpha\right)  $ and $\left(
B,\beta\right)  $ are disjoint relative to their fixed point algebras.
\end{proof}

In order to illustrate Theorem 3.3 by an example satisfying all its
assumptions, we however also need a nontrivial coupling. In the next example
we consider a case of Corollary 5.8 where we are able to construct such a coupling.

\begin{example}
Again we use ideas from Section 3 and the notation in Example 2.4. Write
$A=A_{\theta}$, but with $\theta$ irrational to ensure $A$ is nuclear, and let
$\mu$ be its canonical trace. From the modular conjugation $J$ of 
$\left(A^{\prime\prime},\mu\right)$, we obtain an idempotent map $j$ as in 
Section 3. Setting
\[
\tilde{u}:=j(u)\text{ and }\tilde{v}:=j(v)
\]
we define $\tilde{A}$ to be the C*-subalgebra of $B(H)$ generated by $\left\{
\tilde{u},\tilde{v}\right\}  $, and we naturally carry $\mu$ to the state
$\tilde{\mu}$ on $\tilde{A}$ given by $\tilde{\mu}:=\mu\circ j$, which turns
out to be exactly $\tilde{\mu}=\left\langle \Omega,\left(  \cdot\right)
\Omega\right\rangle $. So $A$ and $\tilde{A}$ are mutually commuting, allowing
us as in Section 3 to define a unital $\ast$-homomorphism $\delta
:A\otimes\tilde{A}\rightarrow B(H)$ by extending $\delta\left(  a\otimes
b\right)  :=ab$, which gives the coupling $\kappa$ of $\left(  A,\mu\right)  $
and $\left(  \tilde{A},\tilde{\mu}\right)  $ defined by
\[
\kappa(a):=\left\langle \Omega,\delta(a)\Omega\right\rangle
\]
for all $a\in A\otimes\tilde{A}$.

This abstract approach actually has a simple concrete meaning. It is easily
verified that $J$ is just complex conjugation on $H=L^{2}\left(
\mathbb{T}^{2}\right)  $, and that $\tilde{u}$ and $\tilde{v}$ are the
generators of $A_{-\theta}$, so they are given by precisely the same formulas
as $u$ and $v$ respectively in Example 2.4, but with $\theta$ replaced by
$-\theta$.

So we simply apply Corollary 5.8 to the case $\theta_{1}=\theta=-\theta_{2}$
to obtain disjoint $\left(  A,\alpha\right)  $ and $\left(  B,\beta\right)  $,
assuming $p/c$ is irrational. In terms of the notation in Corollary 5.8's
proof, we can define a state $\omega$ on $A\otimes B$ by
\[
\omega(a):=\left\langle \Omega,\delta\circ E(a)\Omega\right\rangle
\]
for all $a\in A\otimes B$. Setting $R:=C^{\ast}(v)\otimes C^{\ast}(z)$ and
$\psi:=\omega|_{R}$, we have $\kappa|_{R}=\psi$, and by the above mentioned
disjointness, $\omega$ is easily checked to be the unique joining of 
$\left(A,\alpha,\mu\right)$ and $\left(B,\beta,\nu\right)$ extending $\psi$, 
where $\mu$ and $\nu$ are the canonical traces of $A$ and $B$ as usual.

Keeping in mind Remark 4.4, we therefore satisfy all the assumptions in
Theorem 3.3 in a nontrivial way, since $\kappa\neq\omega$, for example
$\kappa\left(  u\otimes\tilde{u}^{-1}\right)  =1$ while $\omega\left(
u\otimes\tilde{u}^{-1}\right)  =0$.
\end{example}

\section*{Acknowledgments}

We thank Joe Diestel for very helpful discussions regarding the Bochner
integral. This work was supported by the National Research Foundation of South Africa.

\end{document}